\documentclass{paper}

\usepackage{arxiv}

\usepackage[utf8]{inputenc} 
\usepackage[T1]{fontenc}    
\usepackage{amsmath,amsthm}
\usepackage{hyperref}       
\usepackage{array}                    
\usepackage{amsfonts}       
\usepackage{nicefrac}       
\usepackage{mathrsfs}
\numberwithin{equation}{section}
\newtheorem{thm}{Theorem}[section]
\newtheorem{prop}[thm]{Proposition}
\newtheorem{lem}[thm]{Lemma}
\newtheorem{cor}[thm]{Corollary}
\theoremstyle{definition}
\newtheorem{dfn}[thm]{Definition}
\newtheorem{rem}[thm]{Remark}
\newtheorem*{rem*}{Remark}
\newtheorem{exmp}[thm]{Example}
\newtheorem*{exmp*}{Example}

\newcommand{\K}{\mathbb{K}}                             
\newcommand{\N}{\mathbb{N}}                             
\newcommand{\R}{\mathbb{R}}                             
\newcommand{\Rbar}{\overline{\R}}												
\newcommand{\mcN}{\mathcal{N}}

\newcommand{\rkl}[1]{\left( #1 \right)}									
\newcommand{\ekl}[1]{\left[ #1 \right]}									
\newcommand{\gkl}[1]{\left\{ #1 \right\}}								
\newcommand{\spr}[1]{\left\langle #1 \right\rangle}			
\newcommand{\norm}[1]{\left\|{#1}\right\|}							
\newcommand{\abs}[1]{\left|{#1}\right|}									

\DeclareMathOperator{\dom}{dom}                         
\DeclareMathOperator{\hyph}{-}													

\title{A note on $\Gamma$-convergence of Tikhonov functionals \\for nonlinear inverse problems}

\author{
  Alexey Belenkin \\
  Department of Numerical Mathematics\\
  Saarland University\\
  Saarbr\"ucken, Germany \\
  \texttt{s9albele@stud.uni-saarland.de} \\
       \And
  Michael Hartz \\
  Department of Mathematics\\
  Saarland University\\
  Saarbr\"ucken, Germany \\
  \texttt{hartz@math.uni-sb.de} \\
       \And
  Thomas Schuster \\
  Department of Numerical Mathematics\\
  Saarland University\\
  Saarbr\"ucken, Germany \\
  \texttt{thomas.schuster@num.uni-sb.de} \\
}

\begin{document}
\maketitle

\begin{abstract}
We consider variational regularization of nonlinear inverse problems in Banach spaces using Tikhonov functionals. This article addresses the problem of $\Gamma$-convergence of a family of Tikhonov functionals and assertions of the convergence of their respective infima. Such questions arise, if model uncertainties, inaccurate forward operators, finite dimensional approximations of the forward solutions and / or data, etc. make the evaluation of the original functional impossible and, thus, its minimizer not computable. But for applications it is of utmost importance that the minimizer of the replacement functional approximates the original minimizer. Under certain additional conditions this is satisfied if the approximated functionals converge to the original functional in the sense of $\Gamma$-convergence. We deduce simple criteria in different topologies which guarantee $\Gamma$-convergence as well as convergence of minimizing sequences. 
\end{abstract}


\smallskip


\smallskip

\section{Introduction}\label{sec:intro}

Nonlinear inverse problems arise in a large variety of applications ranging from medicine, natural science and engineering. Such problems are usually modeled as operator equations
\begin{equation} \label{Fxy}
	F(x)=y
\end{equation}
for a continuous, nonlinear mapping $F:\dom(F)\subseteq X\to Y$, where $\dom(F)$ denotes the domain of $F$ and $X$, $Y$ are Banach spaces. Problems like \eqref{Fxy} usually are ill-posed in the sense that they are not continuously invertible. The stable solution of such equations by regularization methods has been widely studied over the last decades and published in textbooks such as 
\cite{engl, groetsch1984theory, louisbook, RIEDER:03, schuster2012regularization}. An important class of regularization methods relies on the minimization of \emph{Tikhonov functionals}
\begin{equation} \label{eq_Tikhonov}
	T(x) := \frac{1}{p}\norm{F(x)-y^\delta}_Y^p + \alpha \, \Omega(x)
\end{equation}
in $\dom(F)$, which is well-known as \emph{variational regularization}. Here $y^\delta\in Y$ denotes noise-contaminated data with noise level $\delta>0$, i.e., $\norm{y^\delta - y}_Y\leq \delta$, $\Omega:\dom(F)\to \R_+$ is a functional acting as penalty term and $\alpha>0$ serves as regularization parameter and balances the weighting of the data fitting term $\norm{F(x)-y^\delta}_Y^p$ and the penalty term $\Omega(x)$. The existence of minimizers, the stability and regularization property of $T$ and different parameter choice rules for $\alpha$ depending on $\delta$ and / or $y^\delta$ have been studied in various settings and are very well understood.

In practical applications and implementations of variational regularization methods however $T(x)$ as well as its minimizer(s) are barely accessible, since approximations in finite dimensional subspaces of $X$, $Y$, respectively, e.g. by discretization, are necessary or the forward problem $F(x)$ is too expensive or just impossible to be evaluated exactly. This is, e.g., the case if only a viscosity or weak solution of a corresponding PDE are well defined rather than a classical one (see, e.g., \cite{VIERUS;SCHUSTER:22} for an application in tensor tomography). A further problem of computing $T(x)$ or $F(x)$ exactly are model uncertainties that have to be taken into account (see \cite{RIEDER:03}). That means, instead of computing a minimizer of $T$ in applications one rather computes a minimizer of
\begin{equation} \label{Tikhonov_delta}
	T_\delta(x) := \frac{1}{p} \norm{F_{\delta}(x)-y^\delta}_Y^p + \alpha(\delta) \, \Omega(x)
\end{equation}
for some $\delta>0$ including noise as well as model inaccuracies $F_\delta$, which are the result of processes mentioned above. In some situations it might be more convenient to write $T_n$ instead of $T_\delta$, e.g. to model a discrete approximation process. Standard theory in inverse problems tries to find parameter choice rules $\alpha(\delta)$ such that a minimizer $x_\delta$ of $T_\delta$, if existent, converges to a solution $x^\dagger$ of \eqref{Fxy}. Articles such as \cite{PLATO;VAINIKKO:90, poschl2010discretization} investigate the combination of regularization and discretization. In general the analysis for such convergence results demands for specific asymptotic behaviors of the noise level, the model or discretization error, or assumptions on the nonlinearity of the problem such as the tangential cone condition, all of which are difficult or impossible to check for specific applications. In this article we rather address the question under which conditions we can guarantee that a minimizer $x_\delta$ is \emph{close} to a minimizer of $T$ in the sense that
\begin{equation} \label{inf_property}
	\min_{x\in\dom(F)} T(x) = \lim_{\delta\to 0} \inf_{x\in\dom(F)} T_\delta(x) \; ,   
\end{equation}
where we replace $\delta\to 0$ by $n \to \infty$ if we want to address a discrete process. Such assertions can be achieved by proving that $T_\delta \to T$ as $\delta\to 0$ ($T_n\to T$ as $n\to \infty$) in the sense of \emph{$\mathit{\Gamma}$-convergence}. So we are more interested in the question, whether $x_\delta$ is a good replacement for a (potentially) not accessible minimizer of $T$ rather than convergence to an exact solution of \eqref{Fxy}, a question which is of large interest from a practical point of view. For this purpose $\Gamma$-convergence seems to be superior to pointwise or uniform convergence of $T_\delta$ to $T$.

The concept of $\Gamma$-convergence was introduced by Ennio De Giorgi in the 1970s and can be seen as a generalization of G-convergence, a notion of convergence for Green functions. $\Gamma$-convergence is often used for homogenization problems, dimension reduction, and transitions from discrete to continuous models because it helps to capture the relevant behavior of solutions of a family of minimum problems 
\[
  \min\gkl{T_\varepsilon(u);\, u\in X_\varepsilon}
\]
in a problem 
\[
  \min\gkl{T(u);\, u\in X}
\]
that does not depend on the parameter $\varepsilon$ and typically presents a limiting case for the parameter $\varepsilon$. Standard references on $\Gamma$-convergence are the textbooks \cite{Braides2005, DalMaso1993}. The following properties of $\Gamma$-convergence make this idea especially intriguing (see \cite{Braides2005}): 
\begin{itemize}
  \item the limit functional is always lower semicontinuous; 
	\item the convergence is stable under additive continuous perturbations; 
	\item under some suitable conditions on the the family $\gkl{T_\varepsilon}$ like, e.g., equi-coercivity, also the convergence of their minimum problems is guaranteed, that is
    \begin{itemize}
      \item the limit functional has a minimum point,
      \item the infima of $\gkl{T_\varepsilon}$ converge to the minimum value of $T$, and
      \item the minimizers of $\gkl{T_\varepsilon}$ converge (up to subsequences) to a minimizer of $T$.
    \end{itemize}
\end{itemize}
The last point can be understood as an extension of the following result from calculus of variations to families of functionals: A coercive and lower semicontinuous functional attains its infimum and has a minimum point. A very prominent example, where $\Gamma$-convergence arises in image processing, is the approximation of the Mumford-Shah-functional by the Ambrosio-Tortorelli-functional, see \cite{AMBROSIO1990TORTORELLI}. In this article a fairly general setting is considered, existing results are more specific. For weak-to-weak sequentially continuous, exact forward operators similar results were achieved in \cite{schuster2012regularization}, but not in the setting of $\Gamma$-convergence. In \cite{burger2021variational} $\Gamma$-convergence with respect to the weak$^*$ topology was mentioned for exact bounded linear operators. The outcome in this work goes beyond these results.

\subsubsection*{Outlook and Main Results}

In Section \ref{sec:Preliminaries}, we give a short introduction to the concept of $\Gamma$-convergence. This includes convergence results for minima and minimizers of a sequence of functionals, which are presented in \cite{Braides2005} and \cite{DalMaso1993}. 
We also recall a few basics from the theory of Banach spaces. In Section \ref{sec:GammaConvergenceCoercivenessTikhonovFunctionals}, we define the functionals in question and propose conditions that will allow the theory of $\Gamma$-convergence to be applied to those functionals. We work in a Banach space setting and consider the norm, weak, and weak$^*$ topologies to highlight the differences between the respective conditions in these topologies. 
In Section \ref{sec:Examples} we present examples to demonstrate  how the theoretical results from the previous section apply to typical inverse problems settings like integral equations of first kind and a parameter identification problem for an elliptic boundary value problem. The main achievements of this article are the presentation of criteria for a family of Tikhonov functionals in a fairly general setting to satisfy $\Gamma$-convergence to some limit functional as well as an infimal property as \eqref{inf_property}.


\section{Preliminaries}\label{sec:Preliminaries}
In this section, we give the necessary basics of the theory of $\Gamma$-convergence and also list a few results from functional analysis and topology about compactness in Banach spaces that will be needed later on. 

\subsection{Banach spaces} \label{subsec:CompactnessBanachSpaces}

The Tikhonov functionals we study are defined on (subsets of) Banach spaces. 
In this section, we collect some basic facts from functional analysis that we require. 
For further background, we refer the reader to the books \cite{Conway90,Megginson1998,Rudin91}.

Every Banach space is naturally equipped with the norm topology. 
However, it is known from variational minimization that compactness plays an important role for minimization problems. 
Since norm compact sets are scarce in infinite dimensions, we consider in addition to the norm topology two other topologies on a Banach space $X$. 
These are the weak topology, and, provided that $X$ is a dual space, the weak$^*$ topology.

First, we recall the definition of the weak topology.

\begin{dfn}
Let $X$ be a normed space over $\K = \R$ or $\K = \mathbb{C}$. 
The \emph{(continuous) dual space}  of $X$ is the space
\begin{equation*}
	X^* = \gkl{x^*:X \to \K ; x^* \text{ is linear and continuous}},
\end{equation*}
equipped with the norm
\begin{equation*}
	\norm{x^*} = \sup \gkl{ \abs{x^*(x)} ; x \in X \text{ with } \norm{x} \leq 1 }.
\end{equation*}
The \emph{weak topology} on $X$ is the coarsest topology on $X$ for which all elements of $X^*$ are continuous.
\end{dfn}

The weak topology on $X$ can be understood as the locally convex topology on $X$ induced by the semi-norms
\begin{equation*}
  p_{x^*}(x) = \abs{x^*(x)} \quad (x \in X),
\end{equation*}
where $x^*$ ranges over all elements on $X^*$. 
As usual, topological properties that hold with respect to the weak topology are said to hold \emph{weakly}. 
For more explanation of this construction, we refer the reader to \cite[Chapter 5]{Conway90} or \cite[Chapter 2]{Megginson1998}.

The weak topology is at most coarser than the norm topology, and hence may admit more compact sets. 
This principle works best in reflexive spaces, whose definition we now recall.

\begin{dfn}
Let $X$ be a normed space. The dual space $(X^*)^*$ of $X^*$ is called the \emph{bidual} of $X$ and is denoted by $X^{**}$. The map 
\[
  i:X\to X^{**}, \quad i(x)(x^*)= x^*(x) \quad (x \in X, x^* \in X^*),
\]
is called the \emph{canonical embedding} from $X$ into $X^{**}$. 
A normed space $X$ is \emph{reflexive} if the canonical embedding $i:X\to X^{**}$ is surjective.
\end{dfn}

Examples of reflexive spaces are $L^p$ spaces for $1 < p < \infty$, see, for instance, \cite[Theorem 1.11.10]{Megginson1998}.

The importance of reflexivity for compactness in the weak topology is explained by the following theorem.
\begin{thm} \label{thm:ReflexiveCompact} 
A normed space is reflexive if and only if its closed unit ball is compact in the weak topology.
\end{thm}
\begin{proof}
Cf. \cite[2.8.2 Theorem, p.245]{Megginson1998}.
\end{proof}

We also mention in passing another important result about weak compactness, namely the Eberlein--\v{S}mulian theorem. 
It shows in particular that a subset of a Banach space is weakly compact if and only if it is sequentially weakly compact, see, for instance, \cite[Theorem 2.8.6]{Megginson1998}.

Theorem \ref{thm:ReflexiveCompact} shows that in non-reflexive spaces such as $L^\infty$ spaces, the weak topology is less useful. When working with a dual space, this issue can sometimes be circumvented by using the weak$^*$ topology, whose definition is given as follows.

\begin{dfn}
Let $X$ be a normed space. The weak$^*$ topology on $X^*$ is the coarsest topology on $X^*$ for which the evaluation maps
\begin{equation*}
	X^* \to \K, \quad x^* \mapsto x^*(x),
\end{equation*}
are continuous for all $x \in X$.
\end{dfn}

The weak$^*$ topology on $X^*$ can be understood as the locally convex topology on $X^*$ induced by the semi-norms
\begin{equation*}
  p_{x}(x^*) = \abs{x^*(x)} \quad (x^* \in X^*),
\end{equation*}
where $x$ ranges over all elements on $X$. 
For more explanation of this construction, we once again refer to \cite[Chapter 5]{Conway90} or \cite[Chapter 2]{Megginson1998}.

The crucial result about compactness with respect to the weak$^*$ topology is the following result, which is a generalization of one implication in Theorem \ref{thm:ReflexiveCompact}.
\begin{thm}[Banach--Alaoglu] \label{thm:Banach-Alaoglu} 
Let $X$ be a normed space. Then the closed unit ball of the dual space $X^*$ is compact in the weak$^*$ topology.
\end{thm}
\begin{proof}
Cf. \cite[2.6.18 Theorem, p.229]{Megginson1998}.
\end{proof}

The weak and weak$^*$ topologies are never metrizable in the context of infinite dimensional spaces. 
However, if $X$ is separable, then the unit ball of $X^*$ is metrizable in the weak$^*$ topology; 
see for example \cite[Theorem V.5.1]{Conway90}. Moreover, if $X$ is separable and reflexive, then the unit ball of $X$ is metrizable in the weak topology; see the discussion following \cite[Theorem V.5.1]{Conway90}.

Finally, we mention the following basic consequence of the Hahn--Banach theorem.

\begin{thm} \label{thm:Convex->NormClosure=WeakClosure} 
If $M$ is a convex subset of a normed space, then its norm closure coincides with its weak closure. 
In particular, $M$ is weakly closed if and only if it is norm closed.
\end{thm}
\begin{proof}
Cf. \cite[2.5.16 Theorem, p.216]{Megginson1998}. 
\end{proof}

The analogue of Theorem \ref{thm:Convex->NormClosure=WeakClosure} for the weak$^*$ topology is false 
if $X$ is not reflexive. 
Indeed, if $X$ is not reflexive, then the kernel of any element of $X^{**} \setminus i(X)$ is a norm closed subspace of $X^*$ that is not weak$^*$ closed
(see \cite[Theorem IV.3.1 and Theorem V.1.3]{Conway90}).

\subsection{Semi-continuity and coercivity}

In this subsection, we recall the notions of semi-continuity and coercivity of functionals on topological spaces. These guarantee the existence of minima and will be used throughout the paper.

In the following we denote the set of the extended real numbers by $\Rbar:=\ekl{-\infty,+\infty}:=\R\cup\gkl{+\infty,-\infty}$. 

\begin{dfn} \label{dfn:LowerSemiContinuous} 
Let $X$ be a topological space, $x\in X$, and let $\mcN(x)$ denote the set of all open neighborhoods of $x$ in $X$. A functional $f:X\to\Rbar$ is said to be \emph{lower semicontinuous at $x\in X$}, if for every $t\in\R$, with $t<f(x)$ , there exists $U\in\mcN(x)$ such that $t < f(y)$ for every $y\in U$. We say that $f$ is \emph{lower semicontinuous (l.s.c)} on $X$ if $f$ is lower semicontinuous at each point $x\in X$.
\end{dfn}

The notion of upper semicontinuity is obtained by replacing $<$ with $>$ in the previous definition.

If $X$ is a metric space, then a function $f: X \to \R$ is lower semicontinuous at $x \in X$ if and only if
\[
	f(x) \leq \liminf_{j\to\infty} f\rkl{x_j} \; ,
\]
for every sequence $\rkl{x_j}_{j\in\N}$ converging to $x$ in $X$.

We will make use of the following well known result.

\begin{lem} \label{lem:norm_weakly_lsc}
Let $X$ be a normed space.
\begin{enumerate}
	\item The norm on $X$ is weakly lower semicontinuous.
	\item The norm on $X^*$ is weak$^*$ lower semicontinuous.
\end{enumerate}
\end{lem}

\begin{proof}
See, for instance, \cite[Theorem 2.5.21]{Megginson1998} and \cite[Theorem 2.6.14]{Megginson1998}.
\end{proof}

The following concept is crucial in the calculus of variations. 
Recall that a subset $A$ of a topological space is countably compact if every countable open cover of $A$ has a finite subcover. Clearly, every compact set is countably compact.

\begin{dfn} \label{dfn:Coerciveness} 
Let $X$ be topological space. We say that a functional $f:X\to\Rbar$ is \emph{coercive} on $X$ if the closure of its sublevel set $\gkl{f\leq t}:=\gkl{x\in X; f(x)\leq t}$ is countably compact for every $t\in\R$. This is equivalent to the existence of closed countably compact sets $K_t\subseteq X$ such that $\gkl{f\leq t}\subseteq K_t$ for every $t\in\R$. A functional $f:X\to\Rbar$ is \emph{mildly coercive} if there exists a non-empty countably compact set $K\subseteq X$ such that $\inf_{x\in X} f(x) = \inf_{x\in K} f(x)$.
\end{dfn}

\begin{rem} \label{rem:IntermediateCoerciveness} 
If $f:X\to\Rbar$ is coercive, then $f$ is also mildly coercive. In fact, if $f\neq +\infty$, then there exists $t\in\R$ such that $\gkl{f\leq t}$ is not empty, and we can take $K$ as the closure of this set in $X$. In the case $f= +\infty$ we may take any countably compact subset of $X$ as $K$. The converse is in general not true. An example of a non-coercive, mildly coercive function is given by any periodic function $f:\R^n\to\R$.
\end{rem}

The following theorem is a well known result from the calculus of variations.

\begin{thm} \label{thm:MinimizerConvergenceCoerciveLSC} 
Let $X$ be a topological space. If $f:X\to\Rbar$ is coercive and lower semicontinuous, then
\begin{enumerate}
\item $f$ has a minimum point in $X$;
\item if $\rkl{x_j}_{j\in\N}$ is a minimizing sequence of $f$ in $X$ and $x$ is a cluster point of $\rkl{x_j}_{j\in\N}$, then $x$ is a minimum point of $f$ in $X$;
\item if $f$ is not identically $+\infty$, then every minimizing sequence for $f$ has a cluster point.
\end{enumerate}
\end{thm}
\begin{proof}
Cf. \cite[Theorem 1.15, p.13]{DalMaso1993}.
\end{proof}

\subsection{\texorpdfstring{$\Gamma$}{[Gamma]}-convergence} \label{subsec:GammaConvergence} 

Here we give a short introduction to $\Gamma$-convergence following the book by Dal Maso \cite{DalMaso1993}. We also list some useful properties for applied problems. Another reference is the book by Braides \cite{Braides2005}.

\begin{dfn} \label{dfn:GammaLim} 
Let $\rkl{X,\tau}$ be a topological space, $x\in X$, and $\rkl{f_j}_{j\in\N}$ a sequence of functionals with $f_j: X\to\Rbar$ for all $j\in\N$. The \emph{$\mathit{\Gamma}$-lower limit} and \emph{$\mathit{\Gamma}$-upper limit} of the sequence $\rkl{f_j}_{j\in\N}$ at $x$ are defined by 
\[
	\rkl{\Gamma\hyph\liminf_{j\to\infty} f_j}(x) = \sup_{U\in\mcN(x)} \liminf_{j\to\infty} \inf_{y\in U} f_j(y) \; ,
\]
\[
	\rkl{\Gamma\hyph\limsup_{j\to\infty} f_j}(x) = \sup_{U\in\mcN(x)} \limsup_{j\to\infty} \inf_{y\in U} f_j(y) \; ,
\]
where $\mcN(x)$ is the set of all open neighborhoods of $x$ in $X$.
If there exists a functional $f_\infty: X\to\Rbar$ such that 
\[
	\rkl{\Gamma\hyph\liminf_{j\to\infty} f_j}(x) = \rkl{\Gamma\hyph\limsup_{j\to\infty} f_j}(x) = f_\infty(x) 
\]
for all $x\in X$, then we write $f_\infty = \Gamma\hyph\smash[b]{\lim\limits_{j\to\infty}} f_j$ and we say that the sequence $\rkl{f_j}_{j\in\N}$ \emph{$\mathit{\Gamma}$-converges} to $f_\infty$ (in $X$) or that $f_\infty$ is the \emph{$\mathit{\Gamma}$-limit} of $\rkl{f_j}_{j\in\N}$ (in $X$).
\end{dfn}

\begin{rem}
If $X$ is a metric space, then $\rkl{f_j}_{j\in\N}$ $\Gamma$-converges to $f_\infty:X\to\Rbar$ if and only if the following conditions are satisfied:
\begin{enumerate}
\item for every $x\in X$ and for every sequence $\rkl{x_j}_{j\in\N}$ in $X$ converging to $x$, the inequality
\[
	f_\infty(x) \leq \liminf_{j\to\infty} f_j\rkl{x_j} \;
\]
holds, and
\item for every $x\in X$ there exists a sequence $\rkl{x_j}_{j\in\N}$ in $X$ converging to $x$ such that
\[
	f_\infty(x) \geq \limsup_{j\to\infty} f_j\rkl{x_j} \; .
\]
\end{enumerate}
A proof of the equivalence can be found in \cite[Proposition 8.1, p.86]{DalMaso1993}. 
It turns out that in our setting, the characterization in Definition \ref{dfn:GammaLim} is somewhat easier to deal with. In addition, using Definition \ref{dfn:GammaLim} allows us to deal with weak and weak$^*$ topologies, which are typically not metrizable.
\end{rem}

\begin{rem} \label{rem:SequentialSpacesGammaLim} 
Actually, the above equivalence, as well as all other results from \cite[Proposition 8.1]{DalMaso1993}, also hold true in a more general class of spaces, called sequential spaces. These are topological spaces whose topology is given by sequentially open sets. This generalization can be easily proven by using the fact that every sequential space is a quotient of some first-countable space, which was shown in \cite{Franklin1965}.
\end{rem}

\begin{rem} \label{rem:GammaLimLSC&ConstSeqGC} 
The $\Gamma$-lower limit and the $\Gamma$-upper limit of a sequence of functionals $\rkl{f_j}_{j\in\N}$ from a topological space $X$ into $\Rbar$ are both lower semicontinuous on $X$; see \cite[Proposition 6.8]{DalMaso1993}.
\end{rem}

It is helpful to recall the relationship between $\Gamma$-convergence and pointwise or uniform convergence.

\begin{prop} \label{rem:PointwiseUniformGammaConvergence} 
Let $X$ be a topological space and let $\rkl{f_j}_{j\in\N}$ be a sequence of functionals from $X$ into $\Rbar$. 
\begin{enumerate}
\item The following inequalities hold: 
\[
	\Gamma\hyph\liminf_{j\to\infty} f_j \leq \liminf_{j\to\infty} f_j \; ,\quad \Gamma\hyph\limsup_{j\to\infty} f_j \leq \limsup_{j\to\infty} f_j \; .
\]
In particular, if $\rkl{f_j}_{j\in\N}$ $\mathit{\Gamma}$-converges to $f_\infty$ and converges pointwise to $f$, then $f_\infty\leq f$. 
\item If each $f_j$ is lower semicontinuous and $\rkl{f_j}_{j\in\N}$ converges uniformly to $f$, then $\rkl{f_j}_{j\in\N}$ $\Gamma$-converges to $f$.
\end{enumerate}
\end{prop}

\begin{proof}
See, for instance, \cite[Proposition 5.1]{DalMaso1993} for part (a) and \cite[Proposition 5.2, Remark 5.3]{DalMaso1993} for part (b).
\end{proof}

We will now give the central theorem that makes $\Gamma$-convergence useful and important for applications. For this we need following definitions.

\begin{dfn} \label{dfn:EquiCoerciveness} 
Let $X$ be a topological space. We say that a sequence $\rkl{f_j}_{j\in\N}$ of functionals from $X$ into $\Rbar$ is \emph{equi\nobreakdash-coercive} if for every $t\in\R$ there exists a closed countably compact set $K_t\subseteq X$ of $X$ such that $\gkl{f_j\leq t}\subseteq K_t$ for every $j\in\N$. The sequence $\rkl{f_j}_{j\in\N}$ is \emph{equi-mildly coercive} if there exists a non-empty countably compact set $K\subseteq X$ such that $\inf_{x\in X} f_j(x) = \inf_{x\in K} f_j(x)$ for all $j\in\N$. 
\end{dfn}

\begin{dfn} \label{dfn:MinimizingSequenceForSequenceOfFunctionals} 
Let $f:X\to\Rbar$ be a functional and let $\varepsilon>0$. An \emph{$\varepsilon$-minimizer} of $f$ in $X$ is a point $x\in X$ such that
\[
	f(x) \leq \max\gkl{\inf_{y\in X} f(y) + \varepsilon \, ,\, -\frac{1}{\varepsilon}} \; .
\]
\end{dfn}

The following theorem can be seen as a generalization of Theorem \ref{thm:MinimizerConvergenceCoerciveLSC} to sequences of functionals.

\begin{thm} \label{thm:MinimizerConvergenceGC} 
Let $X$ be a topological space and let $\rkl{f_j}_{j\in\N}$ be a sequence of functionals from $X$ into $\Rbar$.
Suppose that $\rkl{f_j}_{j\in\N}$ is equi-mildly coercive and $\mathit{\Gamma}$-converges to $f_\infty$.
\begin{enumerate}
  \item The functional $f_\infty$ attains its minimum on $X$ and 
\[
	\min_{x\in X} f_\infty (x) = \lim_{j\to\infty} \inf_{x\in X} f_j (x) \; .
\]
\item
Let $\rkl{x_j}_{j\in\N}$ be a sequence in $X$ such that $x_j$ is an $\varepsilon_j$-minimizer of $f_j$ in $X$ for every $j\in\N$, where $\rkl{\varepsilon_j}_{j\in\N}$ is a sequence of positive real numbers converging to 0. Then any cluster point $x_\infty$ of $\rkl{x_j}_{j\in\N}$ is a minimizer of $f_\infty$ in $X$, and 
\[
	f_\infty (x_\infty) = \lim_{j\to\infty} f_j \rkl{x_j} \; .
\]
\item If moreover $\rkl{f_j}_{j\in\N}$ is equi-coercive and $f_\infty$ is not identically $+\infty$,
  then every sequence $\rkl{x_j}_{j\in\N}$ as in (b) has a cluster point.
\end{enumerate}
\end{thm}
\begin{proof}
Part (a) is proved in \cite[Theorem 7.4, p.69]{DalMaso1993}. 
	
In the setting of (b), \cite[Corollary 7.20, p.81]{DalMaso1993} shows that any cluster point $x_\infty$ of $\rkl{x_j}_{j\in\N}$ is a minimizer of $f_\infty$. 
Moreover, either 
\begin{align*}
	-\infty < \inf_{x \in X} f_j(x) &\leq f_j\rkl{x_j} \leq \inf_{x \in X} f_j(x) + \varepsilon_j \\
\intertext{or }
	-\infty = \inf_{x \in X} f_j(x) &\leq f_j{x_j} \leq -\frac{1}{\varepsilon_j}
\end{align*}
for all $j\in\N$. By part (a), both sides of these inequalities converge to $\min_{x \in X} f_\infty(x) = f_\infty\rkl{x_\infty}$.

In the setting of (c), let $t = \min_{x \in X} f_\infty(x) + 1$ if $\min_{x \in X} f_\infty(x) \neq -\infty$ and let $t\in\R$ be arbitrary otherwise. 
By equi-coercivity, there exists a closed countably compact set $K \subseteq X$ such that $\gkl{f_j \leq t} \subseteq K$ for all $j\in\N$. 
Part (a) implies that for all but finitely many $j\in\N$, we have $x_j \in \gkl{f_j \leq t} \subseteq K$; whence $\rkl{x_j}_{j\in\N}$ has a cluster point.
\end{proof}

The notion of $\Gamma$-convergence can be extended to families indexed by a continuous parameter in a straightforward manner, as given in \cite[Section 1.9, p.37]{Braides2005}. Moreover, many examples for application of $\Gamma$-convergence are given in \cite{Braides2005}. \\
A more detailed investigation and deeper results, such as the relation between $\Gamma$-convergence and G-convergence or topological set convergence in the sense of Kuratowski, can be found in the book by Dal Maso \cite{DalMaso1993}.

\section{\texorpdfstring{$\mathbf{\Gamma}$}{[Gamma]}-convergence and coerciveness of Tikhonov functionals}\label{sec:GammaConvergenceCoercivenessTikhonovFunctionals}

In this section, we consider the Tikhonov functional arising from the equation $F(x)=y$. As usually only inexact data $y^\delta$ are available instead of $y$, we will be approximating the exact Tikhonov functional by Tikhonov functionals arising from the inexact data and a suitable sequence of operators. Our aim is to apply Theorem \ref{thm:MinimizerConvergenceGC} and derive convergence of minimum values and minimum points of inexact functionals toward the minima and minimizers of the exact functional, respectively. To this end, we postulate conditions under which we will be able to prove $\Gamma$-convergence and equi-mild coercivity of the functionals. Since both concepts depend on the topology of the underlying space, we will consider the most commonly used ones, that is the norm topology, the weak topology, and the weak$^*$ topology. At the end we will compare the different conditions for those topologies. 

\subsection{The setting} \label{subsec:Setting}
Let $X,Y$ be Banach spaces. On $X$, we will consider three different topologies: the norm topology, the weak topology, and, provided that $X$ is a dual space, the weak$^*$ topology.
Let $\tau$ be any of these three topologies.

We assume the following setup:
\begin{itemize}
  \item Let $F:X\to Y$ be an operator (not necessarily linear) with domain $\emptyset\neq\dom(F)\subseteq X$
    such that for all $y \in Y$, the function
    \begin{equation*}
      X \to \R \cup \gkl{+\infty}, \quad x \mapsto
      \begin{cases}
        \norm{F(x) - y}_Y &, \text{ if } x \in \dom(F) \; , \\
        + \infty &, \text{ otherwise,}
      \end{cases}
    \end{equation*}
    is $\tau$-lower semicontinuous on $X$ and norm upper semicontinuous on $\dom(F)$.
    
  \item Let $\Omega:X\to \R \cup \gkl{+\infty}$ be a functional with the following properties: 
	\begin{itemize}
    \item its effective domain $\dom\rkl{\Omega} = \gkl{x \in X ; \Omega(x) < \infty}$ satisfies
      $\dom\rkl{\Omega} \supseteq \dom(F)$;
    \item $\Omega$ is $\tau$-lower semicontinuous on $X$ and norm-upper semicontinuous on $\dom\rkl{\Omega}$.
	\end{itemize}
\end{itemize}

\begin{rem}
Note that the two norm-upper semicontinuity assumptions can be replaced by norm continuity since $\tau$-lower semicontinuity implies norm-lower semicontinuity. 
Moreover, if $\dom(F)$ is $\tau$-closed, then $\tau$-lower semicontinuity of $x \mapsto \norm{F(x) - y}_Y$ on $\dom(F)$ implies $\tau$-lower semicontinuity on $X$ of the first function above.
\end{rem}

Let $y \in Y$ and $\alpha \in [0,\infty)$. Let $p \in [1,\infty)$ and define the target Tikhonov functional as
\[
  T: X \to \R \cup \gkl{+\infty}, \quad x \mapsto 
  \begin{cases}
    \frac{1}{p} \norm{F(x) - y}_Y^p + \alpha \, \Omega(x) &, \text{ if } x \in \dom(F) \; , \\
    + \infty &, \text{ otherwise} \; .
  \end{cases}
\]

We define approximations to $T$ in the following way. 
\begin{itemize}
  \item 
For $n\in\N$, let $F_n: \dom\rkl{F_n} \to Y$ be operators with $\dom\rkl{F_n} \subseteq \dom(F)$ for all $n\in\N$.
\item 
Let the sequence of domains $\rkl{\dom\rkl{F_n}}_{n\in\N}$ be increasing, i.e. $\dom\rkl{F_n}\subseteq\dom\rkl{F_{n+1}}$ for all $n\in\N$, with the property that
$\bigcup_{n=1}^\infty \dom\rkl{F_n}$ is norm dense in $\dom(F)$. 
\item Assume that $\rkl{F_n}_{n\in\N}$ converges to $F$ locally uniformly with respect to $\tau$
  on the $\tau$-closure $\overline{\dom(F)}^{\tau}$, meaning that for each $x\in\overline{\dom(F)}^\tau$ there exists a $\tau$-open neighborhood $U$ of $x$ such that 
\[
  \lim_{n\to\infty} \sup_{x' \in U \cap \dom\rkl{F_n}} \norm{F_n\rkl{x'} - F\rkl{x'}}_Y =0 \; .
\]
(Note that by the density assumption on $\bigcup_{n=1}^\infty \dom\rkl{F_n}$, the set $U \cap \dom\rkl{F_n}$ is not empty for sufficiently large $n$.)
\end{itemize}

Moreover, let $\rkl{y_n}_{n\in\N}$ be a sequence in $Y$ with $\lim_{n\to\infty} \norm{y_n - y}_Y = 0$ and let $\rkl{\alpha_n}_{n\in\N}$ be a sequence in $\rkl{0,\infty}$ converging to $\alpha$. 
We define the approximating Tikhonov functionals by 
\[
	T_n: X \to \R \cup \gkl{+\infty}, \quad x \mapsto 
	\begin{cases}
		\frac{1}{p} \norm{F_n(x) - y_n}_Y^p + \alpha_n \, \Omega(x) &, \text{ if } x\in \dom\rkl{F_n} \; , \\
		+\infty &, \text{ otherwise} \; .
	\end{cases}
\]

\begin{rem}
Let $X$ be equipped with the weak topology. 
Theorem \ref{thm:Convex->NormClosure=WeakClosure} shows that if $\dom(F)$ is convex, then $\dom(F)$ is weakly closed if and only if it is norm closed. 
Moreover, Lemma \ref{lem:norm_weakly_lsc} implies that if $\dom(F)$ is weakly closed and $F$ is weak-to-weak continuous, then the map $x \mapsto \norm{F(x) - y}_Y$ is weakly lower semicontinuous. 
Again by Theorem \ref{thm:Convex->NormClosure=WeakClosure}, if each $\dom\rkl{F_n}$ is convex, then so is the increasing union $\bigcup_{n\in\N} \dom\rkl{F_n}$, hence this union is norm dense in $\dom(F)$ if and only if it is weakly dense.
\end{rem}

\subsection{\texorpdfstring{$\Gamma$}{[Gamma]}-convergence of Tikhonov functionals} \label{subsec:GammaConvergenceTF} 
In this part, we will prove that $\Gamma\hyph\lim_{n\to\infty} T_n=T$ under suitable assumptions, that is we will prove that both the $\Gamma$-lower and $\Gamma$-upper limits from Definition \ref{dfn:GammaLim} of our approximating Tikhonov functionals are equal to the target Tikhonov functional. As it turns out, the definition in terms of neighborhoods is slightly easier to deal with in our setting and allows us to handle all three mentioned topologies in a similar way. 

\begin{thm} \label{thm:ThreeTopGammaConvTF}
Assume the setup of Subsection \ref{subsec:Setting}. 
Then $\rkl{T_n}_{n\in\N}$ $\mathit{\Gamma}$-converges to $T$ with respect to $\tau$.
\end{thm}
\begin{proof}
In each part, we have to show the lim\hspace{0.1em}inf inequality $T \leq \Gamma\hyph\liminf_{n\to\infty} T_n$ and the lim\hspace{0.1em}sup inequality $\Gamma\hyph\limsup_{n\to\infty} T_n \leq T$. We start by showing the lim\hspace{0.1em}sup inequality in the norm topology. Since every open neighborhood in the weak or the weak$^*$ topology is also an open neighborhood in the norm topology, this also establishes the lim\hspace{0.1em}sup inequality in the weak and in the weak$^*$ topology, see \cite[Proposition 6.3]{DalMaso1993}.

To show the lim\hspace{.1em}sup inequality in the norm topology, let $x \in X$. If $x\notin\dom(F)$, then $T(x) = \infty$, so the inequality trivially holds. Let $x\in\dom(F)$. For $k\in\N$ with $k\geq1$, let $B_{1/k}(x)$ be the open ball with radius $\frac{1}{k}$ centered at $x$. 
In the definition of the $\Gamma$-upper limit, it suffices to consider open neighborhoods of $x$ of the form $B_{1/k}(x)$, that is,
\[
  \Gamma\hyph\limsup_{n\to\infty} T_n(x) = \lim_{k\to\infty} \limsup_{n\to\infty} \inf_{x' \in B_{1/k}(x)} T_n\rkl{x'} \; ;
\]
see for instance \cite[Remark 4.3]{DalMaso1993}. 
Let $k\in\N$ with $k\geq1$. Since $\bigcup_{n=1}^\infty \dom\rkl{F_n}$ is norm dense in $\dom(F)$, there exist $n_0\in\N$ and $\xi_k \in B_{1/k}(x) \cap \dom\rkl{F_{n_0}}$. If $n \geq n_0$, then $\xi_k \in \dom\rkl{F_n}$, so the convergence assumption on $\rkl{F_n}_{n\in\N}$, which in particular implies pointwise convergence on $\bigcup_{n=1}^\infty \dom\rkl{F_n}$, yields $\lim_{n\to\infty} \norm{F_n\rkl{\xi_k} - F\rkl{\xi_k}}_Y = 0$. Hence
\[
  \lim_{n\to\infty} T_n\rkl{\xi_k} = T \rkl{\xi_k}
\]
and so
\[
  \limsup_{n\to\infty} \inf_{x' \in B_{1/k}(x)} T_n(x) \leq \limsup_{n\to\infty} T_n \rkl{\xi_k} = T \rkl{\xi_k} \; .
\]
Since $\lim_{k\to\infty} \norm{\xi_k - x}_X = 0$ and $\xi_k \in \dom(F)$ for all $k$, we find that $\limsup_{k\to\infty} T \rkl{\xi_k} \leq T(x)$ and so
\[
  \Gamma\hyph\limsup_{n\to\infty} T_n(x) = \lim_{k\to\infty} \limsup_{n\to\infty} \inf_{x' \in B_{1/k}(x)} T_n(x) \leq \limsup_{k\to\infty} T \rkl{\xi_k} \leq T(x) \; ,
\]
which completes the first step.

It remains to show the lim\hspace{0.1em}inf inequality in all three cases. 
For all $n\in\N$ define
\begin{align*}
  S_n: X \to \R \cup \gkl{+\infty}, & \quad x \mapsto
  \begin{cases}
    \frac{1}{p} \norm{F_n(x) - y_n}_Y^p &, \text{ if } x \in \dom\rkl{F_n} \; , \\
    + \infty &, \text{ otherwise} \; ,
  \end{cases}
\intertext{and }
  R_n: X \to \R \cup \gkl{+\infty}, & \quad x \mapsto
	\begin{cases}
		\alpha_n \, \Omega(x) &, \text{ if } x \in \dom\rkl{\Omega} \; , \\
		+\infty &, \text{ otherwise} \; .
	\end{cases}
\end{align*}
Then $T_n \geq S_n + R_n$ for all $n\in\N$ and so by a basic inequality for the $\Gamma\hyph\liminf$ of a sum  (see \cite[Proposition 6.17]{DalMaso1993}), we find that
\[
  \Gamma\hyph\liminf_{n\to\infty} T_n \geq \Gamma\hyph\liminf_{n\to\infty} \rkl{S_n + R_n} \geq \Gamma\hyph\liminf_{n\to\infty} S_n + \Gamma\hyph\liminf_{n\to\infty} R_n \; .
\]
To deal with the second summand, observe that since $\rkl{\alpha_n}_{n\in\N}$ tends to $\alpha$, we have
\[
  \rkl{\Gamma\hyph\liminf_{n\to\infty} R_n}(x) = \sup_{U \in \mcN(x)} \liminf_{n\to\infty} \inf_{x' \in U} R_n\rkl{x'}
  = \alpha \, \sup_{U \in \mcN(x)} \inf_{x' \in U} \Omega\rkl{x'}
  = \alpha \, \Omega(x) \; ,
\]
where that last equality follows from the lower semicontinuity of $x \mapsto \Omega(x)$.

It remains to show that
\[
  \rkl{\Gamma\hyph\liminf_{n\to\infty} S_n}(x) \geq
	\begin{cases}
		\frac{1}{p} \norm{F(x) - y}_Y^p &, \text{ if } x \in \dom(F) \; , \\
		+ \infty &, \text{ otherwise} \; .
	\end{cases}
\]
Since the function $t \mapsto \frac{1}{p} t^p$ is continuous and increasing, it suffices to consider the case $p=1$, as the $\Gamma$-lower limit is preserved by applying continuous increasing functions, see \cite[Proposition 6.16]{DalMaso1993}. 
If $x$ is not in the $\tau$-closure of $\dom(F)$, then there exists a $\tau$-open neighborhood $U$ of $x$ that has an empty intersection with the closure of $\dom(F)$, and so $\Gamma\hyph\liminf_{n\to\infty} {T_n}(x) = + \infty$, so the inequality holds in this case.

Let $x$ be in the $\tau$-closure of $\dom(F)$, $\varepsilon > 0$, and $t\in\R$ with $\norm{F(x)-y}_Y>t$. 
Here, we use the convention $\norm{F(x)-y}_Y=+\infty$ for $x\notin\dom(F)$. For all $x' \in \dom\rkl{F_n}$ and all $n\in\N$, the triangle inequality yields
\[
  \norm{F_n\rkl{x'} - y_n}_Y \geq \norm{F\rkl{x'} - y}_Y - \rkl{\norm{F_n\rkl{x'} - F\rkl{x'}}_Y + \norm{y - y_n}}_Y \; .
\]
The convergence assumptions on $\rkl{F_n}_{n\in\N}$ and on $\rkl{y_n}_{n\in\N}$ show that there exist $n_0\in\N$ and an open neighborhood $V_1$ of $x$ such that for all $n \geq n_0$ and all $x' \in V_1 \cap \dom\rkl{F_n}$, each of the two summands in parentheses is less than $\varepsilon$. By lower semicontinuity of $x \mapsto \norm{F(x) - y}_Y$, we may further find another open neighborhood $V_2$ of $x$ and achieve that
\[
  \norm{F\rkl{x'} - y}_Y > t
\]
for all $x' \in V_2 \cap \dom(F)$. Therefore, setting $V:= V_1 \cap V_2$, we have
\begin{align*}
  \Gamma\hyph\liminf_{n\to\infty} S_n(x)
  &= \sup_{U \in \mcN(x)} \liminf_{n\to\infty} \inf_{x' \in U \cap \dom\rkl{F_n}} \norm{F_n\rkl{x'} - y_n}_Y \\
  &\geq \liminf_{n\to\infty} \inf_{x' \in V \cap \dom\rkl{F_n}} \norm{F_n\rkl{x'} - y_n}_Y \\
  &\geq t - 2\, \varepsilon \; .
\end{align*}
Since $\varepsilon > 0$ and $t<\norm{F(x)-y}_Y$ were arbitrary, the desired inequality follows.
\end{proof}

\begin{rem} \label{rem:WeakerConditionForThreeTopGCTF}
Examination of the proof shows that the assumption of locally uniform convergence of $\rkl{F_n}_{n\in\N}$ can be weakened to
\[
  \inf_{U \in \mcN(x)} \limsup_{n\to\infty} \sup_{x' \in U \cap \dom\rkl{F_n}} \norm{F_n\rkl{x'} - F\rkl{x'}}_Y = 0 \quad \text{ for all } x \in \dom(F) \; .
\]
Equivalently, for all $x\in\dom(F)$ and all $\varepsilon > 0$, there exist an open neighborhood $U$ of $x$ and $n_0\in\N$ such that
\[
  \norm{F_n\rkl{x'} - F\rkl{x'}}_Y < \varepsilon \quad \text{ for all } n \geq n_0 \text{ and all } x' \in U \cap \dom\rkl{F_n} \; .
\]
This differs from local uniform convergence in that the neighborhood $U$ may depend on $\varepsilon$.
\end{rem}

We require the following elementary lemma.
\begin{lem} \label{lem:ScaledGammaLim}
Let $\rkl{f_j}_{j\in\N}$ be a sequence of functionals from a topological space $X$ onto $\Rbar$ that $\mathit{\Gamma}$-converges to \\$f: X \to \Rbar$. 
Let $\rkl{\lambda_j}_{j\in\N}$ be a sequence in $\rkl{0,\infty}$ converging to $\lambda\in\ekl{0,\infty}$. 
Then 
\begin{equation*}
	\Gamma\hyph\lim_{j\to\infty} \rkl{\lambda_j \, f_j}(x) = \lambda \, f(x)
\end{equation*}
whenever the right-hand side is not of the form $0 \cdot \pm \infty$ or $\infty \cdot 0$.
\end{lem}
\begin{proof}
We use the following basic fact: 
If $\rkl{\phi_j}_{j\in\N}$ is a sequence in $\Rbar$, then
\begin{equation} \label{eqn:liminf_ineq}
	\liminf_{j\to\infty} \lambda_j \, \phi_j \geq \lambda \, \liminf_{j\to\infty} \phi_j
\end{equation}
whenever the right-hand side is defined. 
Indeed, if $-\infty < t < \liminf_{j\to\infty} \phi_j$, then $\phi_j \geq t$ for all but finitely many $j\in\N$ and so
\begin{equation*}
	\liminf_{j\to\infty} \lambda_j \, \phi_j \geq \liminf_{j\to\infty} \lambda_j \, t = \lambda \, t,
\end{equation*}
provided the right-hand side is not of the form $\infty \cdot 0$. 
Distinguishing the cases $\lambda = 0, \lambda = +\infty$ and $\lambda\in\rkl{0,\infty}$, Inequality \eqref{eqn:liminf_ineq} follows.
  
Let $x \in X$ and $U\in\mcN(x)$. 
Then Inequality \eqref{eqn:liminf_ineq} shows that
\begin{equation} \label{eqn:liminf_2}
	\liminf_{j\to\infty} \inf_{x' \in U} \rkl{\lambda_j \, f_j}\rkl{x'} \geq \lambda \, \liminf_{j\to\infty} \inf_{x' \in U} f_j\rkl{x'},
\end{equation}
provided the right-hand side is defined. 
The definition of the $\Gamma$-lower limit shows that if $f(x) \in \R \setminus \gkl{0}$, then there exists $V\in\mcN(x)$ with $\liminf_{j\to\infty} \inf_{x' \in U} f_j\rkl{x'} \in \R \setminus \gkl{0}$ for all $U\in\mcN(x)$  with $U \subseteq V$. 
So if $\lambda \, f(x)$ is defined, then \eqref{eqn:liminf_2} holds for all $U\in\mcN(x)$ with $U \subseteq V$. Taking the supremum over all such $U$ in \eqref{eqn:liminf_2} gives
\begin{equation*}
	\Gamma\hyph\liminf_{j\to\infty} \rkl{\lambda_j \, f_j}(x) \geq \lambda \, f(x).
\end{equation*}
  
A similar argument shows that
\begin{equation*}
	\Gamma\hyph\limsup_{j\to\infty} \rkl{\lambda_j \, f_j}(x) \leq \lambda \, f(x)
\end{equation*}
when the right-hand side is defined. Combining both inequalities gives the result.
\end{proof}

\begin{cor} \label{cor:InverseWeightedGammaConvTF}
Assume the setup of Subsection \ref{subsec:Setting} and that $\alpha>0$. Then $\rkl{\frac{1}{\alpha_n} T_n}_{n\in\N}$ $\mathit{\Gamma}$-converges to $\frac{1}{\alpha} T$ in the topology $\tau$.
\end{cor}

If $\alpha=0$, then we have the following convergence instead:

\begin{prop} \label{prop:InverseWeightedGammaConvTF0}
Assume the conditions of Theorem \ref{thm:ThreeTopGammaConvTF} and let $\alpha = 0$. If additionally we have that 
\[
	\lim_{n\to\infty}\frac{\norm{y_n - y}_Y}{\alpha_n^{1/p}} = 0 \; ,
\]
and
\[
  \lim_{n\to\infty} \frac{\norm{F_n(x) - F(x)}_Y}{\alpha_n^{1/p}} = 0 
\]
for each $x \in F^{-1}(y)$, then $\rkl{\frac{1}{\alpha_n}T_n}_{n\in\N}$ $\mathit{\Gamma}$-converges to
\[
	\tilde{T}: X \to \R \cup \gkl{+\infty}, \quad x \mapsto 
  \begin{cases}
    \Omega(x) &, \text{ if } x \in F^{-1}(y) \; , \\
    + \infty &, \text{ otherwise} \; ,
	\end{cases}
\]
in the considered topology $\tau$.
\end{prop}
\begin{proof}
For any $x\notin F^{-1}(y)$ we have that $T(x)\neq0$, so Lemma \ref{lem:ScaledGammaLim} implies that $\Gamma\hyph\lim_{n\to\infty} \frac{1}{\alpha_n} T_n (x) = +\infty$ for $x\notin F^{-1}(y)$.

Let now $x\in F^{-1}(y)$. The lim\hspace{0.1em}sup inequality $\Gamma\hyph\limsup_{n\to\infty} \frac{1}{\alpha_n}T_n \leq \tilde{T}$ follows from 
\begin{align*}
	\rkl{\Gamma\hyph\limsup_{n\to\infty} \frac{1}{\alpha_n}T_n}(x) 
	&= \sup_{U \in \mcN(x)} \limsup_{n\to\infty} \inf_{x' \in U \cap \dom\rkl{F_n}} \frac{1}{p\ \alpha_n} \norm{F_n\rkl{x'} - y_n}_Y^p + \Omega\rkl{x'} \\
	&\leq \limsup_{n\to\infty} \frac{1}{p\ \alpha_n} \norm{F_n(x) - y_n}_Y^p + \Omega(x) \\
	&= \frac{1}{p} \rkl{\limsup_{n\to\infty} \frac{\norm{F_n(x) - F(x)}_Y}{\alpha_n^{1/p}} + \frac{\norm{y - y_n}_Y}{\alpha_n^{1/p}}}^p + \Omega(x) \\
	&= \Omega(x) \; .
\end{align*}
The lim\hspace{0.1em}inf inequality $\tilde{T} \leq \Gamma\hyph\liminf_{n\to\infty} \frac{1}{\alpha_n}T_n$ also follows easily from the lower semicontinuity of $\Omega$:
\begin{align*}
  \Gamma\hyph\liminf_{n\to\infty} \frac{1}{\alpha_n} T_n(x)
  \geq \Gamma\hyph\liminf_{n\to\infty} \Omega(x) = \Omega(x).
\end{align*}
Combining both inequalities yields the result.
\end{proof}

This result is similar to one by Burger, see \cite[Lemma 3.4]{burger2021variational}. However, our initial assumptions on the setting are slightly more general, and therefore his result is a specific case of the proposition above.

\subsection{Equi-mild coercivity of Tikhonov functionals} \label{subsec:EquiMildCoercivenessTF} 
In this part, we will investigate whether our Tikhonov functionals are equi-mildly coercive, that is if there exists a suitable compact set (in the relevant topology) such that all the functionals have their global infimum on that set.

\begin{prop} \label{prop:EquiMildCoercivenessTF}
Assume the setup of Subsection \ref{subsec:Setting} and let $\alpha > 0$. 
In addition, assume that the sublevel sets $\gkl{\Omega\leq t}$ are relatively countably compact in the topology $\tau$ for every $t\in\R$. 
Then $\rkl{T_n}_{n\in\N}$ is equi-coercive in the considered topology $\tau$.
\end{prop}
\begin{proof}
Let $t\in\R$. Since $\alpha>0$, there exists $\delta>0$ with $\alpha_n > \delta > 0$ for all $n\in\N$. Thus, we have $T_n(x) \geq \delta \, \Omega(x)$ for every $x\in\dom\rkl{F_n}$. Hence 
\[
	\gkl{T_n \leq t} \subseteq \gkl{\Omega \leq \frac{t}{\delta}}
\]
for all $n\in\N$. 
Therefore, the equi-coercivity follows from the assumed relative countable compactness of the sublevel sets of $\Omega$ w.r.t. the topology $\tau$.
\end{proof}

\begin{cor} \label{cor:Weak*TopEquiMildCoercivenessTF}
Assume the setup of Subsection \ref{subsec:Setting}, let $\alpha > 0$ and assume that the sublevel sets $\gkl{\Omega \leq t}$ are bounded for every $t\in\R$.
\begin{enumerate}
\item If $X$ is a reflexive Banach space, then $\rkl{T_n}_{n\in\N}$ is equi-coercive in the weak topology.
\item If $X$ is a dual space, then $\rkl{T_n}_{n\in\N}$ is equi-coercive in the weak$^*$ topology.
\end{enumerate}
\end{cor}

\begin{proof}
This is immediate from Proposition \ref{prop:EquiMildCoercivenessTF} and the Banach--Alaoglu theorem (Theorem \ref{thm:Banach-Alaoglu}, see also Theorem \ref{thm:ReflexiveCompact}).
\end{proof}

\begin{rem} \label{rem:EquiMildCoercivenessTF}
If $\alpha = 0$ and if there exists a bounded sequence of minimizers of $\rkl{T_n}_{n\in\N}$, then $\rkl{T_n}_{n\in\N}$ is equi-mildly coercive in the weak$^*$ topology (assuming $X$ is a dual space). This again follows from the Banach--Alaoglu theorem (Theorem \ref{thm:Banach-Alaoglu}).

Moreover, if minimizers are convergent in norm, then $\rkl{T_n}_{n\in\N}$ is equi-mildly coercive in the norm topology. Indeed, if $m_n \in X$ is a minimizer of $T_n$ for all $n\in\N$ such that $\rkl{m_n}_{n\in\N}$ converges to some $m\in X$ in norm, then
\[
  K = \gkl{m_n; n\in\N} \cup \gkl{m}
\]
is norm compact, because a convergent sequence, together with its limit point, forms a compact set in any metric space. Clearly,
\[
  \inf_{x \in K} T_n(x) = \inf_{x \in X} T_n(x) \quad \text{ for all } n\in\N \; .
\]
\end{rem}

\begin{rem} \label{rem:InverseWeightedEquiMildCoercivenessTF}
Proposition \ref{prop:EquiMildCoercivenessTF} is also true for $\rkl{\frac{1}{\alpha_n}T_n}_{n\in\N}$, even for $\alpha=0$, since $\frac{1}{\alpha_n}T_n\geq\Omega$ for all $n\in\N$.
\end{rem}

\section{Examples}\label{sec:Examples}

In this section we aim to illustrate the usefulness of the above theoretical results by applying it to the following examples.

\begin{exmp} 
First we consider an integral equation of first kind with continuous kernel. To this end let $X = Y = Z = L^\infty\rkl{\ekl{0,1}}$, which is the dual space of $L^1\rkl{\ekl{0,1}}$, 
and consider the corresponding weak-$*$ topology on $L^\infty \rkl{\ekl{0,1}}$. 
Let $K: \ekl{0,1} \times \ekl{0,1} \to \R$ be continuous. 
Let 
\begin{align*}
  F: L^\infty\rkl{\ekl{0,1}} \to L^\infty\rkl{\ekl{0,1}}, \quad F(x)(t) &= \int_0^1 K\rkl{s,t} x(s) \, \mathrm{d} s \; , \\
	\Omega: L^\infty\rkl{\ekl{0,1}} \to \ekl{0,\infty}, \quad \Omega(x) &= \norm{x}_{L^\infty} \; .
\end{align*}
Obviously, both mappings are (norm-to-norm) continuous. Furthermore, $F$ is also weak$^*$-to-weak$^*$ continuous, since it is the adjoint operator of the integral operator on $L^1\rkl{\ekl{0,1}}$ with kernel $\tilde{K}\rkl{s,t} := K\rkl{t,s}$. Hence $x\mapsto\norm{F(x)-y}$ and $\Omega$ are both weak$^*$ lower semicontinuous by Lemma \ref{lem:norm_weakly_lsc}. 
Therefore, Theorem \ref{thm:ThreeTopGammaConvTF} applies to this example.
\end{exmp}

\begin{exmp} \label{Example_PI_linear}
Parameter identifications for PDEs represent an important class of inverse problems.
We consider the elliptic boundary value problem
\begin{subequations} \label{elliptic_BVP1}
\begin{eqnarray}
\label{elliptic_PDE1}
	-\Delta u + c\, u &=& f\qquad \mbox{in }D \; ,\\
\label{boundary_values}
	u &=& 0\qquad \mbox{on } \partial D \; .
\end{eqnarray}
\end{subequations}
Here, $D\subset \R^d$ denotes an open, convex, bounded domain with smooth boundary $\partial D$, and $d\in\gkl{1,2,3}$. Equation \eqref{elliptic_PDE1} can be seen as stationary Schr\"odinger equation with potential $c(x)$. The forward (direct) problem means to compute the solution $u(x)$ for given functions $c$ and $f$. The inverse problem consists of calculating the source term $f$ from $u$ with given potential $c$. The mathematical formulation of this inverse problem is
\begin{equation} \label{Example_PI_PDE1}
	F(f) = u^{\mathrm{meas}} \; ,
\end{equation}
where $F: \dom(F)\subseteq X \to Y$ maps the parameter $f$ to the unique (weak) solution $u$ of \eqref{elliptic_BVP1} and $u^{\mathrm{meas}}$ are the measured data. We define $X:=L^2(D)$ and $\dom(F) := \{f \in L^2(D) ; \norm{f}_X \leq \rho\}$,
where $\rho > 0$. To ensure the existence and uniqueness of $u$ we furthermore assume $c \in L^\infty(D)$ and $c \geq 0$ a.e. 
Since $u$ is to be a weak solution of \eqref{elliptic_BVP1} we introduce the Sobolev space $H_0^1(D)$, which is the closure of the space $C_0^\infty(D)$ of infinitely differentiable functions with compact support w.r.t. the $H^1$-norm. As usual, we equip $H_0^1(D)$ with the scalar product
\[
	\spr{u,v}_{H_0^1(D)} := \int_D \nabla u \cdot \nabla v \, \mathrm{d}x \; ,\qquad u,v \in H_0^1(D) \; .
\]
On this space its induced norm is equivalent to the $H^1$-norm.  We note that $H_0^1(D)$ is continuously embedded in $L^2(D)$ (\cite[Theorem 4.12]{adams2003sobolev}). Hence $u$ solves the variational problem
\begin{equation}
	a_c(u,v) = \ell_f(v)\qquad \mbox{for all } v\in H_0^1(D)
\end{equation}
with the symmetric bilinear form $a_c : H^1_0(D)\times H^1_0(D) \to \R$
\[
	a_c(u,v) = \int_D \rkl{ \nabla u\cdot \nabla v + c\, u\, v }\, \mathrm{d}x
\]
and the bounded, linear functional on $H_0^1(D)$
\[
	\ell_f(v) = \int_D f\:\! v \,\mathrm{d}x \; ,\qquad v \in H_0^1(D) \; . 
\]
We have that $a_c$ is $H_0^1$-coercive since
\[
	a_c(u,u) \geq \norm{u}_{H_0^1(D)}^2\qquad \mbox{for all } u\in H_0^1(D) \; .
\]
It is also continuous on $H^1_0(D)\times H^1_0(D)$ since for all $u,v\in H_0^1(D)$ we have 
\begin{align*}
	a_c(u,v) &\leq \abs{\spr{u,v}_{H_0^1(D)}} + \norm{c\, u\, v}_{L^1(D)} \\
	&\leq \norm{u}_{H_0^1(D)} \norm{v}_{H_0^1(D)} + \norm{c}_{L^2(D)} \norm{u\, v}_{L^2(D)} \\
  &\leq \norm{u}_{H_0^1(D)} \norm{v}_{H_0^1(D)} + \norm{c}_{L^2(D)} \norm{u}_{L^4(D)} \norm{v}_{L^4(D)} \\
  &\leq \rkl{1+\norm{c}_{L^2(D)} \, C_S^2} \norm{u}_{H_0^1(D)} \norm{v}_{H_0^1(D)} \; ,
\end{align*}
where $C_S$ is a constant depending only on dimension $d$ and domain $D$. In the last step we used the Sobolev Embedding Theorem, see e.g. \cite[Theorem 4.12]{adams2003sobolev}. 
The Lax--Milgram theorem is thus applicable and guarantees that the linear mapping $A_c : H_0^1(D)\to \rkl{H_0^1(D)}^* = H^{-1}(D)$ given by $A_c(u) := a_c(u,\cdot)$
is continuously invertible. 
Since $f\mapsto \ell_f$ is linear and bounded as a mapping from $L^2 (D)$ to $(H^1_0(D))^*$, this yields continuity of $f \mapsto A_c^{-1}(\ell_f) =u$ as a mapping from $L^2(D)$ to $H^1_0(D)$. 
From Theorem 4 (in \S 6.3) and Theorem 6 (in \S 6.2) of \cite{evans2010partial} we even have that, by our assumptions to $D$, the weak solution $u\in H^2 (D)$ and depends continuously on $f$. 
We define $Y := L^2(D)$. Since the inclusion $H^1_0(D) \subset L^2(D)$ is continuous, the forward operator $F: f \mapsto u$ is therefore norm-to-norm continuous. 
Moreover, since $F$ is linear, it is also weak-to-weak continuous (\cite[Theorem 2.5.11]{Megginson1998}). From this, as well as from norm and weak closedness of $\dom(F)$, it follows that the Tikhonov functional $T: \dom(F)\to \R,$
\begin{equation}
	T(f) := \frac{1}{2}\norm{F(f) - u^{\mathrm{meas}}}^2_{Y} + \frac{\alpha}{2} \norm{f}_X^2 
\end{equation}
satisfies the continuity assumptions of Subsection \ref{subsec:Setting} in the norm and in the weak topology. \\
Minimizing $T$ yields a stable regularization method for \eqref{Example_PI_PDE1}, see, e.g., \cite{engl}. Since in practical applications the boundary value problem \eqref{elliptic_BVP1} has to be solved numerically, only approximations of the forward solver $F(f)$ can be computed, e.g., by using the Finite Element method. This method relies on the Galerkin approximation for the weak solution $u$ by choosing finite dimensional, conformal Finite Element spaces $V_n\subset H^1_0(D)$, $\dim\rkl{V_n}<\infty$, $V_n\subset V_{n+1}$, and 
\[
	\overline{\bigcup_{n\in\N} V_n} = H^1_0(D) \; , 
\]
and subsequently solving the variational problems to find $u_n\in V_n$ such that
\begin{equation} \label{H01bilinear_n}
	a_c(u_n,\varphi) = \ell_f(\varphi)\qquad \mbox{for all } \varphi\in V_n \; .
\end{equation}
Again the symmetry, continuity and coercivity of $a_c$ yields a unique solution $u_n\in V_n$. This leads to the fact that in applications one minimizes the Tikhonov functional $T_n: \dom(F)\to \R,$
\begin{equation}
	T_n(f) := \frac{1}{2}\norm{F_n(f) - u^{\mathrm{meas}}}^2_{Y} + \frac{\alpha}{2} \norm{f}_X^2 \; ,
\end{equation}
where $F_n(f) := u_n$ is the unique solution of \eqref{H01bilinear_n}. 
Standard settings for $V_n$ are piecewise linear B-splines for which the convergence
\begin{equation} \label{linBsplineEst1}
  \norm{u_n-u}_{L^2(D)} \leq C\, n^{-2} \norm{f}_{L^2(D)}
\end{equation}
holds true for $n\to \infty$, where $C>0$ is a constant only depending on $D,d$ and $c$. 
The proof of this estimate relies on the C\'ea-Lemma and the Theorem of Aubin-Nitsche (see \cite{aubin1967behavior, nitsche1968kriterium} and \cite[Folgerung 7.7]{braess2013finite}). 
Higher rates are obtained by using higher order splines and for $u$ being of sufficient regularity. 
Since $\dom\rkl{F_n}=\dom(F)$, the estimate \eqref{linBsplineEst1} and the definition of $\dom(F)$ immediately yield uniform convergence $F_n\to F$ as $n\to\infty$ in $X$.

All prerequisites in Subsection \ref{subsec:Setting} are thus satisfied and Theorem \ref{thm:ThreeTopGammaConvTF} is valid for this example in both topologies, the norm topology and the weak topology. Furthermore, we have that $\rkl{T_n}_{n\in\N}$ is equi-coercive in the weak topology of $X$ (Proposition \ref{prop:EquiMildCoercivenessTF}). This leads to the important fact that from any minimizing sequences $\rkl{f_{k,n}}_{k\in\N}$ of $T_n$ one can construct a weakly convergent minimizing sequence $\rkl{f_{k(n),n}}_{n\in\N}$ of $T$ with a monotonically increasing function~$k(n)$ (by Theorem \ref{thm:MinimizerConvergenceGC}), meaning that any minimizer of $T_n$ is in this sense close to a minimizer of $T$, which is very important from a practical point of view.
\end{exmp}

\section{Conclusion} \label{sec:Conclusion}

In this article we presented simple criteria like local uniform convergence and equi-mild coercivity to obtain $\Gamma$\nobreakdash-convergence for families of Tikhonov functionals for nonlinear, continuous (maybe ill-posed) operator equations in Banach spaces and convergence of their corresponding minimal sequences. The topologies addressed in this article for convergence and continuity are the norm, weak and weak$^*$ topologies. The results are of importance for practical applications, where an exact evaluation of the original Tikhonov functional is not possible, since, e.g., the forward operator can not be exactly evaluated or forward solutions and / or measurement data are represented in finite dimensional subspaces. Then our results guarantee that, under mild conditions, the minimizer of the approximated functional is close to a minimizer of the original functional.


\bibliography{NGCTFNLIP_ref} 

\end{document}